\documentclass{amsart}

\usepackage{amsmath,amssymb,amsfonts}
\usepackage[all]{xy}
\usepackage{rotating}

\newcommand{\NS}{{\rm NS}}
\newcommand{\Trans}{{\rm Trans}}
\newcommand{\Inv}{{\rm T}}
\newcommand{\Orth}{{\rm S}}

\newcommand{\diag}{{\rm diag}}

\newcommand{\Tr}{{\rm Tr}}
\newcommand{\Norm}{{\rm N}}

\newcommand{\Aut}{{\rm Aut}}

\DeclareMathOperator{\rank}{rank}
\DeclareMathOperator{\id}{id}
\DeclareMathOperator{\Hom}{Hom}

\newcommand{\coloneqq}{:=}

\newcommand{\IC}{\mathbb{C}}

\newcommand{\IN}{\mathbb{N}}
\newcommand{\IP}{\mathbb{P}}
\newcommand{\IR}{\mathbb{R}}
\newcommand{\IQ}{\mathbb{Q}}
\newcommand{\IZ}{\mathbb{Z}}

\newcommand{\cM}{\mathcal{M}}

\newcommand{\lra}{\longrightarrow}

\theoremstyle{plain}
\newtheorem{theorem}{Theorem}[section]
\newtheorem{lemma}[theorem]{Lemma}
\newtheorem{proposition}[theorem]{Proposition}
\newtheorem{corollary}[theorem]{Corollary}
\theoremstyle{definition}

\theoremstyle{remark}
\newtheorem{remark}[theorem]{Remark}
\newtheorem{example}[theorem]{Example}

\begin{document}

\title[Isometries of ideal lattices and hyperk\"ahler manifolds]{Isometries of ideal lattices \\and hyperk\"ahler manifolds}

\author[S. Boissi\`ere]{Samuel Boissi\`ere}
\address{Samuel Boissi\`ere, Laboratoire de Math\'ematiques et Applications, UMR CNRS 6086, Universit\'e de Poitiers, T\'el\'eport 2, Boulevard Marie et Pierre Curie, F-86962 Futuroscope Chasseneuil}
\email{samuel.boissiere@math.univ-poitiers.fr}
\urladdr{http://www-math.sp2mi.univ-poitiers.fr/$\sim$sboissie/}

\author[C. Camere]{Chiara Camere}
\address{Chiara Camere, Leibniz Universit\"at Hannover,
Institut f\"ur Algebraische Geometrie,
Welfengarten 1
30167 Hannover, Germany} 
\email{camere@math.uni-hannover.de}
\urladdr{http://www.iag.uni-hannover.de/~camere}

\author[G. Mongardi]{Giovanni Mongardi}
\address{Giovanni Mongardi, Universit\'{a} degli studi di Milano, Dipartimento di Matematica, via Cesare Saldini 50 20133 Milano, Italy }
\email{giovanni.mongardi@unimi.it}

\author[A. Sarti]{Alessandra Sarti}
\address{Alessandra Sarti, Laboratoire de Math\'ematiques et Applications, UMR CNRS 6086, Universit\'e de Poitiers, T\'el\'eport 2, Boulevard Marie et Pierre Curie, F-86962 Futuroscope Chasseneuil}
\email{alessandra.sarti@math.univ-poitiers.fr}
\urladdr{http://www-math.sp2mi.univ-poitiers.fr/$\sim$sarti/}

\date{\today}

\subjclass{Primary 14J50; Secondary 14C50,55T10}

\keywords{holomorphic symplectic manifolds, non-symplectic automorphisms, ideal lattices, cyclotomic fields}

\begin{abstract} We prove that there exists a holomorphic symplectic manifold deformation equivalent to the Hilbert scheme of two points on a K3 surface that admits a non-symplectic automorphism of order $23$, that is the maximal
possible prime order in this deformation family. The proof uses the theory of ideal lattices in cyclomotic fields.
\end{abstract}

\maketitle


\section{Introduction}\label{s:intro}

The study of automorphisms on deformation families of hyperk\"ahler manifolds is a recent and very active 
 field of research. One of the main objective in the recent published papers concerns the classification 
of prime order automorphisms: fixed locus, moduli spaces and deformations. We refer for instance to \cite{BCS,Mongardi,OW} and references therein for a more complete picture. The purpose of this paper is to answer a question
of \cite{BNWS} concerning the existence of automorphisms of order $23$.

Let $X$ be an irreducible holomorphic symplectic manifold. Its second cohomology group $H^2(X,\IZ)$
is an integral lattice for the Beauville--Bogomolov--Fujiki quadratic form ~\cite{Beauville}. Let $f$ be a holomorphic
automorphism of $X$ of prime order $p$ acting non-symplectically: $f$ acts on $H^{2,0}(X)$ by multiplication
by a primitive $p$-th root of the unity. Such automorphisms can exist only when $X$ is projective. 
It follows that the invariant lattice $\Inv(f)\subset H^2(X,\IZ)$  is a primitive sublattice of the N\'eron--Severi group $\NS(X)$
and consequently the characteristic polynomial of the action of $f$ on the transcendental lattice $\Trans(X)$ is a multiple $k$ of the $p$-th cyclotomic polynomial $\Phi_p$. Thus
$k\varphi(p)=\rank_\IZ\Trans(X)$ and in particular 
$$
\varphi(p)\leq b_2(X)-\rho(X),
$$
where $\varphi$ is the
Euler function and $\rho(X)=\rank_\IZ\NS(X)$ is the Picard number of~$X$. 
Assume that $X$ is in the deformation class of the Hilbert scheme of two points on a projective K3 surface (an IHS-$K3^{[2]}$ for short). Since $b_2(X)=23$,
the maximal order for $f$ is $p=23$ and this can happen only when $\rho(X)=1$. The main result of this paper is:

\begin{theorem} 
There exists an IHS-$K3^{[2]}$ with a non-symplectic automorphism of order $23$. 
\end{theorem}

We show in \S\ref{s:ihs} that the N\'eron--Severi group of $X$ has rank one, generated by an ample line bundle of square $46$ with respect to the Beauville--Bogomolov--Fujiki quadratic form: up to now there does not exist any geometric construction of such an IHS-$K3^{[2]}$ (see \cite{OGrady}). 
We emphasize that such an automorphism can not exist on the Hilbert scheme of two points on a K3 surface since this has
Picard number two.

The strategy of the proof consists
in constructing an isometry of order $23$ of the lattice $E_8^{\oplus 2}\oplus U^{\oplus 3}\oplus\langle -2\rangle$ with the required properties (Corollary ~\ref{cor:isometryL}) and then to
use the surjectivity of the period map and the global Torelli theorem to construct the variety with its automorphism (Theorem~\ref{th:main}).
Assuming that such an automorphism does exist, the invariant lattice $T$ and its orthogonal complement $S$
 are uniquely determined up to isometry so the main step (Proposition~\ref{prop:isometryS}) 
consists in constructing an order $23$ isometry on the lattice $S$: we obtain it by proving that the lattice $S$ can be  realized 
as an ideal lattice in the $23$rd cyclotomic field, using results of Bayer--Fluckiger~\cite{Bayer1,Bayer2,Bayer3}. 

\section{Preliminaries on lattices}\label{s:lattice}

A {\it lattice} $L$ is a free $\IZ$-module equipped with a nondegenerate symmetric bilinear form
$\langle \cdot, \cdot\rangle_L$ with integer values.  Its {\it dual lattice} is 
$L^{\vee}\coloneqq\Hom_{\IZ}(L,\IZ)$. It can be also described as follows:
$$
L^{\vee}\cong\{x\in L\otimes \IQ~|~\langle x,v\rangle_L\in \IZ\quad \forall v\in L\}.
$$
Clearly $L$ is a sublattice of $ L^{\vee}$ of the same rank, so the \emph{discriminant group} ${A_L:=L^{\vee}/L}$ is a finite abelian group  whose order is denoted $d_L$ and called the {\it discriminant of $L$}. In a basis $(e_i)_i$ of~$L$, for the Gram matrix 
$M\coloneqq(\langle e_i,e_j\rangle_L)_{i,j}$ one has $d_L=|\det(M)|$.

A lattice $L$ is called \emph{even} if $\langle x,x\rangle\in 2\IZ$ for all $x\in L$.  In this case   
the bilinear form induces a  quadratic form $q_L: A_L\lra \IQ/2\IZ$. Denoting by $(s_{(+)},s_{(-)})$ the signature of
$L\otimes\IR$, the triple of invariants $(s_{(+)},s_{(-)},q_L)$ characterizes the \emph{genus} of the even lattice $L$ (see \cite{BCS,Dolgachev} and references therein).

A lattice $L$ is called {\it unimodular} if $A_L=\{0\}$. A sublattice $M\subset L$ is called \emph{primitive} if $L/M$ is a free $\IZ$-module.
If $L$ is unimodular and $M\subset L$ is a primitive sublattice, then $M$ and its orthogonal $M^\perp$ in $L$ have
isomorphic discriminant groups and $q_M=-q_{M^\perp}$.

Let $p$ be a prime number. A lattice $L$ is called $p$-\emph{elementary} if $A_L\cong\left(\frac{\IZ}{p\IZ}\right)^{\oplus a}$ for some
non negative integer $a$ (also called the \emph{length} $\ell(A_L)$ of $A$).  We write $\frac{\IZ}{p\IZ}(\alpha)$, $\alpha\in\IQ/2\IZ$ to denote that the quadratic form $q_L$ takes value $\alpha$ on the generator 
of the $\frac{\IZ}{p\IZ}$ component of the discriminant group. Recall that an even indefinite $p$-elementary lattice of rank $r\geq 3$
with $p\geq 3$ is uniquely determined by its signature and discriminant form (see \cite[Theorem~2.2]{BCS}).

\section{Basic results on non-symplectic automorphisms}\label{s:ihs}

From now on, we assume that $X$ is an IHS-$K3^{[2]}$ with a non-symplectic automorphism $f$ of prime order $3\leq p\leq 23$.
The lattice $H^2(X,\IZ)$ has signature $(3,19)$ and is isometric to $L\coloneqq E_8^{\oplus 2}\oplus U^{\oplus 3}\oplus\langle -2\rangle$, where 
$U$ is the unique even unimodular hyperbolic lattice of rank two and $E_8$ is the negative definite lattice associated to the corresponding Dynkin diagram.
We restate in this special case some results of Boissi\`ere--Nieper-Wisskirchen--Sarti~\cite{BNWS}: the case $p=23$ was left apart  since it requires different arguments due to the fact that the ring of integers of the $23$rd cyclotomic field is not a PID, but some basic
facts extend easily.

The automorphism $f$ induces an isometry $g\coloneqq f^\ast$ on $H^2(X,\IZ)$. We denote by $G=\langle g\rangle$ the group generated by $g$ and we put
$$
\tau\coloneqq g-1\in\IZ[G],\quad \sigma \coloneqq 1+g+\cdots+g^{p-1}\in\IZ[G].
$$
One has $\Inv(f)= \ker(\tau)\cap H^2(X,\IZ)$
and we define $\Orth(f)\coloneqq \ker(\sigma)\cap H^2(X,\IZ)$. 

Denote by $\Phi_p\in\IQ[X]$ the $p$-th cyclotomic polynomial.
Consider the cyclotomic field $K=\IQ[X]/(\Phi_p)\cong \IQ(\zeta_p)$ with ring of algebraic integers $O_K\cong\IZ[\zeta_p]$ (here $\zeta_p=X\mod \Phi_p$ should not be considered as a complex number).
The $G$-module structure of $K$ is defined by $g\cdot x=\zeta_p x$ for $x\in K$. For any fractional ideal $I$ in $K$,
and $\alpha\in I$, we denote by $(I,\alpha)$ the module $I\oplus \IZ$ whose $G$-module structure is defined by $g\cdot (x,k)=(\zeta_p x+k\alpha,k)$.
By a theorem of Diederichsen--Reiner~\cite[Theorem~74.3]{CurtisReiner}, $H^2(X,\IZ)$ is isomorphic as a $\IZ[G]$-module
to a direct sum:
$$
(A_1,a_1)\oplus \cdots\oplus (A_r,a_r)\oplus A_{r+1}\oplus\cdots A_{r+s}\oplus Y
$$
for some $r,s\in \IN$, where $A_i$ are fractional ideal in $K$, $a_i\in A_i$ are such that $a_i\notin (\zeta_p-1)A_i$ and $Y$ is a free $\IZ$-module of finite rank on which $G$ acts trivially.

\begin{lemma} \label{lem:ptorsion} The quotient $\frac{H^2(X,\IZ)}{\Inv(f)\oplus\Orth(f)}$ is a $p$-torsion module.
\end{lemma}

\begin{proof} First we observe that $\Inv(f)\cap\Orth(f)=0$ since $H^2(X,\IZ)$ has no $p$-torsion. It is clear that $Y\subset\Inv(f)$ and $O_K\subset\Orth(f)$.
Let $A=\sum_i O_K \alpha_i$ be a fractional ideal of K, with $\alpha_i\in K$. Clearly $A\subset \Orth(f)$.
In any term $(A,a)=A\oplus \IZ$, denoting $v\coloneqq (0,1)$ in this decomposition, we show that
$pv\in \Inv(f)\oplus\Orth(f)$. One has $\tau(pv)=(pa,0)$. Write $a=\sum_i x_i\alpha_i$ with $x_i\in O_K$.
Since $O_K/(\zeta_p-1)\cong \IZ/p\IZ$, there exists $z_i\in O_K$ such that $px_i=(\zeta_p-1)z_i$. Hence $\tau(pv)=((\zeta_p-1)z,0)$ with $z\coloneqq \sum_i z_i\alpha_i\in A$. Now $\tau((z,0))=((\zeta_p-1)z,0)$ hence
$\tau((pv-(z,0))=0$ and $\sigma((z,0))=0$ so finally $pv=(pv-(z,0))+(z,0)\in\Inv(f)\oplus\Orth(f)$.
\end{proof}

We define $a_f\in\IN$ such that 
$$
\frac{H^2(X,\IZ)}{\Inv(f)\oplus\Orth(f)}\cong\left(\frac{\IZ}{p\IZ}\right)^{\oplus a_f}.
$$
By definition, $\Orth(f)$ is a torsion-free $O_K$-module for the action $\zeta_p\cdot x=g(x)$ for all $x\in S(f)$, hence $\Orth(f)_\IQ\coloneqq \Orth(f)\otimes_\IZ\IQ$ is a $K$-vector space. It follows that there
exists $m_f\in\IN^\ast$ such that 
$$
\rank_\IZ\Orth(f)=\dim_\IQ \Orth(f)_\IQ=(p-1)m_f.
$$
It is easy to check that $\Orth(f)$ is the orthogonal complement of $\Inv(f)$ in the lattice $H^2(X,\IZ)$ (see \cite[Lemma~6.1]{BNWS}). By a similar argument as in \cite[Lemma~6.5]{BNWS}
one deduces from Lemma~\ref{lem:ptorsion}  that the invariant lattice $\Inv(f)$ has signature ${(1,22-(p-1)m_f)}$ and discriminant 
$$
A_{\Inv(f)}\cong\left(\frac{\IZ}{2\IZ}\right)\oplus\left(\frac{\IZ}{p\IZ}\right)^{\oplus a_f}
$$ 
and that $\Orth(f)$ has signature $(2,(p-1)m_f-2)$ and discriminant 
$$
A_{\Orth(f)}\cong\left(\frac{\IZ}{p\IZ}\right)^{\oplus a_f}.
$$
If $p\geq 3$, as explained in \cite[proof of Lemma~6.5]{BNWS} the action of $G$ on $A_{\Orth(f)}$ is trivial.
Since $f$ acts non-symplectically one has  $\Trans(X)\subset\Orth(f)$ and $\rank_\IZ\Trans(X)\geq p-1$. In particular, if $m_f=1$ this forces $\Trans(X)=\Orth(f)$ and consequently $\NS(X)=\Inv(f)$. Since $1$ is not an eigenvalue of $\left.f^\ast\right|_{\Orth(f)}$ the characteristic polynomial of $\left.f^\ast\right|_{\Orth(f)}$ is $\Phi_p$.

All the possible isometry classes for the lattices $\Inv(f)$ and $\Orth(f)$ have been classified in \cite{BCS} 
when $2\leq p\leq 19$ (only partially for $p=5$) by using the previous properties (for $p=2$ the situation is a bit different), the Lefschetz fixed point formula and a relation between the cohomology modulo $p$ of the fixed locus and the integers $a_f,m_f$ obtained
using Smith theory methods~\cite{BNWS}.

If $p=23$  the only possibility is that
$m_f=1$, $\Inv(f)$ has signature $(1,0)$ and $\Orth(f)$ has signature $(2,20)$. The case $a_f=0$ is impossible
by Milnor's theorem since there exists no even unimodular lattice with signature $(2,20)$, hence $a_f=1$ since $\Inv(f)$
has rank one, so $A_{\Inv(f)}\cong \frac{\IZ}{46\IZ}$ and finally $\Inv(f)$ is isometric to the lattice $\langle 46\rangle$.
By results of Nikulin and Rudakov--Shafarevich \cite{Nikulin,RS}  $\Orth(f)$ splits 
as a direct sum $U\oplus W$ where $W$ is hyperbolic and $23$-elementary 
of signature $(1,19)$, $A_W\cong\frac{\IZ}{23\IZ}$ and $W$ is unique
up to isometry. 
It follows that $\Orth(f)$ is uniquely determined and  we deduce that $\Orth(f)$ is isometric
to the lattice $U^{\oplus 2}\oplus E_8^{\oplus 2}\oplus K_{23}$
where $K_{23}\coloneqq \left(\begin{matrix} -12 & 1\\1&-2\end{matrix}\right)$.
As a consequence, if there exists
an IHS-$K3^{[2]}$, say $X$, with a non-symplectic automorphism $f$ of order $23$, then necessarily it has $\rho(X)=1$,
$\NS(X)=\Inv(f)=\langle 46\rangle$ and $\Trans(X)=\Orth(f)= E_8^{\oplus 2}\oplus U^{\oplus 2}\oplus K_{23}$.
Such a variety does not belong to any of the known families: Hilbert schemes of points or
moduli spaces of semi-stable sheaves on projective K3 surfaces have Picard number greater than $2$,
Fano varieties of lines on cubic fourfolds are polarized by a class of square $6$ (see \cite[\S~5.5.2]{BNWS} 
and references therein) and similarly the degree of the polarisation is $2$ for double covers of EPW sextics,
it is $38$ for the sums of powers of general cubics of Iliev--Ranestad and it is $22$ for the 
varieties of Debarre--Voisin (see~\cite[\S 0]{OGrady}).

\section{Ideal lattices in cyclotomic fields}
\label{s:ideallattice}

The relation between automorphisms of lattices with given characteristic polynomials and ideals in cyclotomic fields has
been studied by many authors, in particular Bayer-Fluckiger~\cite{Bayer1,Bayer2,Bayer3} and  Gross--McMullen~\cite{GM}. We recall here some results that are needed in the sequel.

Assume that $p$ is an odd prime number. Recall that $K=\IQ(\zeta_p)$ denotes the cyclotomic field with ring of algebraic integers $O_K=\IZ[\zeta_p]$.
We denote respectively by $\Tr_{K/\IQ}$ and $\Norm_{K/\IQ}$ the trace and the norm maps. The complex conjugation on~$K$
is defined as the $\IQ$-linear involution $K\to K,x\mapsto \overline{x}$ such that $\overline{\zeta_p^i}=\zeta_p^{p-i}$ for all $i$.
We denote by $F\subset K$ the \emph{real} subfield of $K$, that is 
$$
F\coloneqq\{x\in K\,|\,\overline{x}=x\}.
$$
Denoting $\mu_p\coloneqq \zeta_p+\zeta_p^{p-1}$ one has
$F=\IQ(\mu_p)$.
The $\IQ$-linear pairing 
$$
(-,-)_K\colon K\times K\to \IQ, \quad (x,y)\mapsto \Tr_{K/\IQ}(x\overline{y})
$$
is non-degenerate and
has determinant $D_K\coloneqq p^{p-2}$ in the basis $(1,\zeta_p,\ldots,\zeta_p^{p-2})$.

Let $(S,\langle-,-\rangle_S)$ be an integral even  lattice of rank $p-1$, signature $(s_{+},s_{-})$ and discriminant $d_S$.
Assume that $S$ admits a non trivial isometry $\varphi$ of order $p$. Its characteristic polynomial is then $\Phi_p$
so $S$ admits a natural structure of $O_K$-module defined by $\zeta_p\cdot x=\varphi(x)$ for all $x\in S$. For dimensional
reasons  $S_\IQ\coloneqq S\otimes_\IZ\IQ$ is isomorphic to $K$ so the inclusion $S\hookrightarrow S_\IQ\cong K$ identifies
the lattice $S$ with an $O_K$-submodule of $K$ (a fractional ideal of $K$), so $S$ becomes an
\emph{ideal lattice} in $K$. Observe that $S_\IQ$ is identified with $K$  in such a way that the 
isometry $\varphi$ corresponds to the 
mutiplication by $\zeta_p$ in $K$. The multiplication by $\zeta_p$ is an isometry for $(-,-)_K$ and also for $\langle-,-\rangle_S$
extended to $S_\IQ$ since it corresponds to the action of $\varphi$. Since the trace is non degenerate, under the identification $S_\IQ=K$ there
exists a unique $\alpha\in K$ such that
$$
\langle x,y\rangle_S=(\alpha x, y)_K\quad \forall x,y\in S
$$
Since the bilinear form on $S$ is symmetric 
one has 
$\overline{\alpha}=\alpha$
so $\alpha\in F$.

If $I\subset K$ is a fractional ideal and $\alpha\in F$, we denote by $I_\alpha$ the ideal lattice whose
bilinear form is $\langle x,y\rangle_\alpha\coloneqq  \Tr(\alpha x\overline{y})$.
Some of the main invariants of the lattice $I_\alpha$ correspond to properties of $\alpha$ 
that we explain now.

Recall that the \emph{norm} of $I$ is defined as $N(I):=|\det(\psi)|$ where $\psi\colon K\to K$ is any $\IQ$-linear 
automorphism such that $\psi(O_K)=I$. By a direct computation one finds that the discriminant $d_{I_\alpha}$ of $I_\alpha$ satisfies the relation:

\begin{align}\label{property1}
d_{I_\alpha}=N(I)^2|N_{K/\IQ}(\alpha)|D_K.
\end{align}
Observe that since $\alpha\in F$ one has 
$$
N_{K/\IQ}(\alpha)=N_{F/\IQ}\left(N_{K/F}(\alpha)\right)=N_{F/\IQ}(\alpha^2)=N_{F/\IQ}(\alpha)^2.
$$

We recall that the \emph{codifferent} $O_K^\vee$ of $K$ is defined by
$$
O_K^\vee\coloneqq \{x\in K\,|\, \forall y\in O_K, \Tr_{K/\IQ}(xy)\in\IZ\}.
$$
If $I_\alpha$ is an integral lattice, for any $x,y\in I$
one has $\Tr_{K/\IQ}(\alpha x\overline{y})\in\IZ$ so $\alpha x \overline{y}\in O_K^\vee$. The integrality of $I_\alpha$ is thus equivalent
to the condition:
\begin{align}\label{property2}
\alpha x\overline{y}\in O_K^\vee\quad \forall x,y\in I.
\end{align}

Note that if $\alpha$ satisfies the above property the lattice $I_\alpha$ is automatically even:
putting $\gamma\coloneqq \sum\limits_{i=0}^{(p-1)/2}\zeta_p^i$, since $\gamma+\overline{\gamma}=1$ one has for any $x\in I$
\begin{align*}
\langle x,x\rangle_S=\Tr_{K/\IQ}(\alpha x\overline{x})=\Tr_{K/\IQ}((\gamma+\overline{\gamma})\alpha x\overline{x})=\Tr_{K/\IQ}(\gamma\alpha x\overline{x})+\Tr_{K/\IQ}(\overline{\gamma\alpha x}x)\in 2\IZ
\end{align*}
since $\alpha x \overline{x}\in O_K^\vee$ by assumption.

The field $K$ admits $p-1$ complex embeddings defined by $\zeta_p\mapsto {\rm{e}}^{\frac{2{\rm{i}}k\pi}{p}}$, $1\leq k\leq p-1$
that induce real embeddings of $F$. We denote by $t$ the number of these real embeddings such that $\alpha$ is negative. One can show that the lattice $I_\alpha$ has signature 
\begin{align}\label{property3}
(p-1-2t,2t).
\end{align}
This is a special case of \cite[Proposition~2.2]{Bayer1}, we recall the argument for convenience. 
First observe that $K$ is a quadratic extension of $F$ with minimal polynomial $X^2-\mu_pX+1\in F[X]$. 
Denoting $\theta\coloneqq \zeta_p^2+\zeta_p^{-2}-2$ one has thus $K\cong F(\sqrt{\theta})$.  
Each complex embedding of $K$ induces a real embedding $v\colon F\to \IR$ such that
$v(\theta)<0$. It follows that $K\otimes_\IQ\IR$ decomposes in a direct sum
$$
K\otimes_\IQ\IR=\bigoplus_{v\colon F\to \IR} \IR\left(\sqrt{v(\theta})\right)
$$
where the sum runs over all real embeddings of $F$.
Each factor is isomorphic to $\IC$ and the complex conjugation on $K$ induces
 the usual complex conjugation on each factor $\IC$. 
On each factor, the form $\langle -,-,\rangle_\alpha$ computed in the $\IR$-basis $(1,\sqrt{v(\theta)})$ 
is $\diag(2v(\alpha),-2v(\alpha)v(\theta))$ so it has signature $(2,0)$ if $v(\alpha)>0$ and signature $(0,2)$ if $v(\alpha)<0$. The result follows.

\section{Construction of isometries of lattices}

We want to determine if a given 
integral even $p$-elementary lattice $S$  of rank $p-1$ with fixed signature and discriminant form admits an
isometry of order $p$ whose characteristic polynomial is the cyclotomic polynomial $\Phi_p$. By the results of Section~\ref{s:ideallattice}, one first has to find an element $\alpha\in F$
satisfying conditions (\ref{property1}),(\ref{property2}),(\ref{property3}). 

\begin{example}
Assume that $p=5$ and $S=U\oplus H_5$ with $H_5=\left(\begin{matrix}2 & 1\\1 &-2\end{matrix}\right)$.
The lattice $S$ is $5$-elementary with $d_S=5$, it has signature $(2,2)$ and discriminant form $A_S=\frac{\IZ}{5\IZ}\left(\frac{2}{5}\right)$.
In \cite[Table~2]{BCS} this case  is denoted by $(p,m,a)=(5,1,1)$ .
In order to recover this lattice as an ideal lattice, since $O_K$ is a PID we take $I=\beta O_K$ for some $\beta\in K$. Equation (\ref{property1}) writes:
$$
N_{K/\IQ}(\beta)^2 N_{F/\IQ}(\alpha)^2 = \frac{1}{5^2}.
$$
Assuming that $\beta=1$, we run a basic computer search program to determine
$\alpha\in~F$ that satisfies all the needed  conditions. Taking $\alpha=\frac{1}{5}(3\mu_5+4)$, in the basis $(1,\zeta_5,\zeta_5^2,\zeta_5^3)$ of~$I$ the bilinear form writes
$$
\left(\begin{matrix}
2& 1&-2&-2\\1&2&1&-2\\-2&1&2&1\\-2&-2&1&2
\end{matrix}\right)
$$
so condition (\ref{property2}) is satisfied and it is easy to check that this lattice has signature $(2,2)$ and discriminant form $\frac{\IZ}{5\IZ}\left(\frac{2}{5}\right)$.
As mentioned above, these invariants characterize the lattice $U\oplus H_5$ up to isometry.
By construction, the order $5$ isometry of this lattice, written in this basis, is the companion matrix of $\Phi_5$:
$$
\left(\begin{matrix}
0&0&0&-1\\1&0&0&-1\\0&1&0&-1\\0&0&1&-1
\end{matrix}\right)
$$
\end{example}

\begin{example}
Assume that $p=13$ and that $S=U^{\oplus 2}\oplus E_8$.
The lattice $S$ is unimodular of signature $(2,10)$.
In \cite[Table~5]{BCS} this case is denoted by $(p,m,a)=(13,1,0)$.
If $S$ admits an order $13$ isometry it induces an identification $S=\beta O_K$ for some $\beta\in K$. Equation (\ref{property1}) writes:
$$
N_{K/\IQ}(\beta)^2 N_{F/\IQ}(\alpha)^2 = \frac{1}{13^{11}}.
$$
It is clear that this equation has no solution, so this lattice does not admit an isometry whose characteristic polynomial
is $\Phi_{13}$. This answers a question left open in \cite[Theorem~7.1]{BCS}: this case cannot be realized by a non-symplectic automorphism of order $13$ on an IHS-$K3^{[2]}$.
\end{example}

We assume now that $p=23$ and we consider the lattice $S\coloneqq  E_8^{\oplus 2}\oplus U^{\oplus 2}\oplus K_{23}$.
It is $23$-elementary with $d_S=23$, signature $(2,20)$ and discriminant form $A_S=\frac{\IZ}{23\IZ}\left(\frac{-2}{23}\right)$. 

\begin{proposition}\label{prop:isometryS}
The lattice $U^{\oplus 2}\oplus E_8^{\oplus 2}\oplus K_{23}$ admits an isometry of 
order~$23$ which acts trivially on the discriminant group $A_S$.
\end{proposition}

\begin{proof}
We apply the strategy developed above. Taking $I=O_K$, equation (\ref{property1}) writes:
$$
N_{F/\IQ}(\alpha) = \frac{1}{23^{10}}.
$$
The software MAGMA~\cite{Magma} provides only one solution to this equation:
$$
\alpha_0\coloneqq \frac{1}{23}(-\mu_{23}^7+\mu_{23}^6+7\mu_{23}^5-6\mu_{23}^4-14\mu_{23}^3+9\mu_{23}^2+7\mu_{23}-2).
$$
This element $\alpha_0$ does not satisfy all the needed solutions so we look for an element $\alpha=\alpha_0\cdot \varepsilon$ with $\varepsilon\in O_F^\ast$: this has the same norm but letting $\varepsilon$ vary this will produce lattices with different signature and discriminant form.
By the Dirichlet unit theorem, the group of units $O_F^\ast$ is the product of the finite cyclic group of roots of unity of $F$ with a free abelian group of rank $10$.
A computation with the software SAGE~\cite{Sage} shows that $O_F^\ast\cong \frac{\IZ}{2\IZ}\times \IZ^{10}$ (the only roots of unity in $F$ are~$\pm 1$) where the free part is generated by the following fundamental units:
\begin{align*}
\epsilon_1&\coloneqq \mu_{23}^4 - 4\mu_{23}^2 + 2\\
\epsilon_2&\coloneqq \mu_{23}^8 - 8\mu_{23}^6 + 20\mu_{23}^4 - 16\mu_{23}^2 + 2\\
\epsilon_3&\coloneqq \mu_{23}^7 - 7\mu_{23}^5 + 14\mu_{23}^3 - 7\mu_{23}\\
\epsilon_4&\coloneqq \mu_{23}^2 - 2\\
\epsilon_5&\coloneqq  \mu_{23}^9 - 9\mu_{23}^7 + 27\mu_{23}^5 - 30\mu_{23}^3 + 9\mu_{23}\\
\epsilon_6&\coloneqq \mu_{23}^9 - 8\mu_{23}^7 + 20\mu_{23}^5 - 16\mu_{23}^3 + \mu_{23}^2 + 2\mu_{23} - 1\\
\epsilon_7&\coloneqq  \mu_{23}\\
\epsilon_8&\coloneqq  \mu_{23}^7 + \mu_{23}^6 -6\mu_{23}^5 - 5\mu_{23}^4 + 10\mu_{23}^3 + 6\mu_{23}^2 - 4\mu_{23} - 1\\
\epsilon_9&\coloneqq  \mu_{23}^8 + \mu_{23}^7 - 8\mu_{23}^6 - 7\mu_{23}^5 + 20\mu_{23}^4 + 14\mu_{23}^3 - 16\mu_{23}^2 - 7\mu_{23}+3\\
\epsilon_{10}&\coloneqq  \mu_{23}^9+\mu_{23}^8-8\mu_{23}^7-7\mu_{23}^6 + 20\mu_{23}^5 + 14\mu_{23}^4 - 16\mu_{23}^3 - 6\mu_{23}^2 +3\mu_{23} - 1
\end{align*}
We consider an element $\alpha\in F$ given as follows:

$$
\alpha=\alpha_0\epsilon_0\epsilon_1^{\nu_1}\cdots\epsilon_{10}^{\nu_{10}}
$$
for some $\nu_i\in\IZ$ with $\epsilon_0\in\{-1,1\}$.
By running a basic computer search program we find that the choice $\epsilon_0=1$,  $\nu=(2, 1, 2, 2, 0, 1, 1, 2, 1, 0)$
satisfies all the needed conditions: in the basis $(1,\zeta_{23},\ldots,\zeta_{23}^{21})$ of~$I$ the bilinear form is the matrix given in Appendix~\ref{app:matrice23}.
It is easy to check that this lattice has signature $(2,20)$ and discriminant form $\frac{\IZ}{23\IZ}\left(\frac{44}{23}\right)$. As already mentioned, these invariants characterize the lattice $U^{\oplus 2}\oplus E_8^{\oplus 2}\oplus K_{23}$ up to isometry.
By construction, the order $23$ isometry of this lattice, written in this basis, is the companion matrix of the polynomial $\Phi_{23}$
and a direct computation shows that that this isometry acts trivially on the discriminant group.
\end{proof}

\begin{corollary}
\label{cor:isometryL}
The lattice $L\coloneqq E_8^{\oplus 2}\oplus U^{\oplus 3}\oplus\langle -2\rangle$ admits an order $23$ isometry whose invariant lattice is isometric to $T\coloneqq\langle 46\rangle$ and
such that the orthogonal complement of $T$ in $L$ is isometric to $S\coloneqq E_8^{\oplus 2}\oplus U^{\oplus 2}\oplus K_{23}$.
\end{corollary}

\begin{proof}
Denoting $T=\IZ t$ with $t^2=46$, the discriminant group $A_T$ is generated by $\tau\coloneqq t/46\in T^\ast$ with $\tau^2=1/46$. We denote by $\sigma\in S^\ast$ a generator of $A_S$
such that $\sigma^2=44/23$. The vector $2\sigma+4\tau\in (S\oplus T)^\ast$ is isotropic in $A_{S\oplus T}$  so it defines an even overlattice
$$
M\coloneqq  S\oplus T\oplus (2\sigma+4\tau)\IZ\subset (S\oplus T)^\ast.
$$
Consider the quotient $H\coloneqq \frac{M}{S\oplus T}\subset A_{S\oplus T}$.
One computes that 
$H^\perp/H$ is generated by the class $23\tau$ with $(23\tau)^2=3/2\in\IQ/2\IZ$. Since $A_M\cong H^\perp/H$ (see \cite{Nikulin}) we conclude that $M$ is an even lattice of signature $(3,20)$ and discriminant form $A_M=\frac{\IZ}{23\IZ}\left(\frac{3}{2}\right)$. By \cite[Theorem~2.2]{Morrison}
these invariants characterize $M$ up to isometry, so $M$ is isomorphic to $L$. It follows directly from the construction that $S$ is the orthogonal of $T$ in $M$.

Let $\varphi$ be the isometry of order $23$ on $S$ constructed in Proposition~\ref{prop:isometryS}. Since $\varphi$~acts trivially on $A_S$, the isometry
$\varphi\oplus \id$ of $S\oplus T$ extends to an isometry on $L$ with the required properties.
\end{proof}

\begin{remark} In the lattice $L$, denoting by $(e,f)$ a basis on one of the factors isometric to the lattice $U$
and by $\delta$ a generator of the factor isometric to $\langle -2\rangle$ it is easy to see that an explicit embedding
of $T$ in $L$ whose orthogonal is isometric to $S$ is given by $t\mapsto 2e+12f+\delta$.
Applying Nikulin's results on primitive embeddings and decomposition of lattices (see~\cite{BCS} and references 
therein) one can see that $T$ admits up to isometry a second embedding in $L$ whose orthogonal complement is isometric to
$E_8^{\oplus 2}\oplus U\oplus\langle -2\rangle\oplus \langle 2\rangle\oplus K_{23}$,
an explicit embedding beeing given by $t\mapsto e+23f$.
\end{remark}

\section{An IHS-$K3^{[2]}$ with a non-symplectic automorphism of order $23$}

\begin{theorem}\label{th:main}
There exists an IHS-$K3^{[2]}$ with a non-symplectic automorphism of order $23$.
This variety $X$ and its automorphism $f$ have the following properties:
\begin{enumerate}
\item $\rho(X)=1$, $\NS(X)\cong\langle 46\rangle$ and $\Trans(X)\cong E_8^{\oplus 2}\oplus U^{\oplus 2}\oplus K_{23}$;
\item  $\Inv(f)=\NS(X)$ and $\Orth(f)=\Trans(X)$. 
\end{enumerate}
\end{theorem}

\begin{proof} The proof is an application of the surjecivity of the period map and of the global Torelli theorem for IHS manifolds.

\par{\it Construction of the variety.} Let $\cM_L^0$ be a connected component of the moduli space of pairs $(X,\eta)$ where $X$ is an IHS-$K3^{[2]}$ and $\eta\colon H^2(X,\IZ)\to L$ is an isometry. The period domain is
$$
\Omega_L\coloneqq \{\omega\in\IP(L_\IC)\,|\, \langle\omega,\omega\rangle_L=0,\langle\omega,\overline{\omega}\rangle_L>0\}.
$$
Recall that the period map $P\colon\cM_L^0\to\Omega_L$ defined by $P((X,\eta))=\eta(H^{2,0}(X))$ is surjective~\cite[Theorem~8.1]{Huybrechts}.

Consider as in Corollary~\ref{cor:isometryL} the embedding of $T=\langle 46\rangle$ in the lattice $L$ whose
 orthogonal complement is $S=E_8^{\oplus 2}\oplus U^{\oplus 2}\oplus K_{23}$, with the isometry $\varphi$ of order $23$ acting trivially on $T$.
We denote by $\omega$ a generator of the one-dimensional eigenspace of $S_\IC$ corresponding to the eigenvalue
$\xi\coloneqq \rm{e}^{\frac{2\rm{i}\pi}{23}}$. Recall that by construction $S$ is identified with the ring of integers $O_K$
of the cyclotomic field $K=\IQ(\zeta_{23})$ so that $\omega\in S_\IC=K\otimes_\IQ\IC$ with basis $(1,\zeta_{23},\ldots,\zeta_{23}^{21})$.
In this basis, the isometry $\varphi$ acts by the companion matrix of the $23$rd cyclotomic polynomial and it is
easy to check that up to a multiplicative constant one has
$$
\omega=\sum_{i=0}^{21}\left(\sum_{j=0}^{21-i}\xi^j\right)\zeta_{23}^{i}.
$$
Since $\varphi(\omega)=\xi\omega$ one has $\langle \omega,\omega\rangle_L=0$. Denoting by $\Tr_{K/\IQ}\colon K_\IC\to\IC$ the
$\IC$-linear extension of the trace, one has
$$
\langle \omega,\overline{\omega}\rangle_L=\langle \omega,\overline{\omega}\rangle_S=\Tr_{K/\IQ}(\omega\alpha\omega)
$$
where $\alpha$ is given in the proof of Proposition~\ref{prop:isometryS}. An explicit computation shows that $\Tr_{K/\IQ}(\omega\alpha\omega)>0$,
so $\omega\in \Omega_L$. 
By surjectivity of the period map, there exists $(X,\eta)\in\cM_L^0$ such that $\eta(H^{2,0}(X))=\omega$.
Then 
$$
\eta(\NS(X))=\left\{\lambda\in L\,|\,\langle \lambda,\omega\rangle_L=0\right\}\supset T.
$$

Let us show that $\eta(\NS(X))=T$. For this, we show that there is no element $\lambda\in S$ with the property that
$\langle\lambda,\omega\rangle_S=0$. In the basis $(1,\zeta_{23},\ldots,\zeta_{23}^{21})$, denoting $\Xi\coloneqq (\xi^{21},\ldots,\xi,1)$
and $J\coloneqq\left(\begin{matrix} 1 & & \text{\huge{$1$}} \\&\ddots &\\ \text{\huge{$0$}}& & 1\end{matrix}\right)$ one has 
by definition $\omega=J\Xi$. Denote by $M$ the matrix
of the lattice $S$ in the basis  $(1,\zeta_{23},\ldots,\zeta_{23}^{21})$ (see Appendix~\ref{app:matrice23}). For any $\lambda\in S$, since 
$S=O_K=\IZ[\zeta_{23}]$ the element $\lambda$ can be identified with a column vector with integer coordinates. Then
$$
\langle \lambda,\omega\rangle_S=\lambda^\top M\omega=\lambda^\top MJ\Xi.
$$
If $\lambda^\top MJ\Xi=0$, since $\lambda^\top MJ$ has integer coordinates and since the coordinates of $\Xi$ are
linearly independent over $\IQ$ it follows that $\lambda^\top MJ=0$. But the matrix $MJ$ is invertible, so $\lambda=0$.
This proves that $\NS(X)\cong T$ and in particular $X$ is projective \cite[Theorem~3.11]{Huybrechts}.

\par{\it Construction of the automorphism.} The isometry $\varphi$ preserves the space $H^{2,0}(X)=\IC\omega$ so it
is a Hodge isometry. Denoting by $q_X$ the Beauville-Bogomolov-Fujiki quadratic form on $H^2(X,\IZ)$, 
the posivite cone $C_X$ is defined as the connected component of the cone 
$$
\{x\in H^2(X,\IZ)\,|\, q_X(x)>0\}
$$
that contains the K\"ahler cone. By Markman~\cite[Lemma~9.2]{Markman} the group of monodromy operators 
of $H^2(X,\IZ)$ is equal to the group of isometries of $H^2(X,\IZ)$ preserving the positive cone $C_X$. 
Here the generator of $\NS(X)\cong T$ is an ample class so it lives in the K\"ahler cone and since
$\NS(X)$ is invariant by $\varphi$ the cone $C_X$ is preserved, so $\varphi$ is a monodromy operator
that leaves invariant a K\"ahler class. By the Global Torelli Theorem of Markman--Verbitsky~\cite[Theorem~1.3]{Markman}
there exists an automorphism $f$ of $X$ such that $f^\ast=\varphi$ on $H^2(X,\IZ)$. Since the natural map
$\Aut(X)\to O(H^2(X,\IZ))$ is injective (see for instance \cite[Lemma~1.2]{Mongardi} and references therein), $f$ is an order $23$ non-symplectic automorphism of $X$.
\end{proof}

\begin{remark}
Since $\NS(X)=\langle 46\rangle$, it follows from \cite[Theorem~2.2]{Markman} that $(X,\eta)$ is the only Hausdorff point in
the fiber $P^{-1}(\omega)$ of the period map so this variety with its automorphism of order $23$ is unique, although it belongs to a $20$-dimensional family of IHS-$K3^{[2]}$ polarized by a class of square $46$.
\end{remark}

\begin{remark}
The same method can be used to produce order $23$ automorphisms on deformations of $K3^{[n]}$ with $n\geq 3$,
under some arithmetic conditions on $n$.
\end{remark}

\appendix

\newpage

\section{Matrix of the lattice with an order $23$ isometry}
\label{app:matrice23}

Here is the matrix of the bilinear form on the lattice $U^{\oplus 2}\oplus E_8^{\oplus 2}\oplus K_{23}$,
written in a basis such that the order $23$ isometry of this lattice is the companion matrix of 
the cyclotomic polynomial $\Phi_{23}$:

\begin{center}
\begin{rotate}{-90}
$
\left(\begin{array}{rrrrrrrrrrrrrrrrrrrrrr}
-2 & 3 & 0 & 3 & 0 & 2 & 1 & 0 & -1
& -2 & -2 & -3 & -3 & -2 & -2 & -1 & 0
& 1 & 2 & 0 & 3 & 0 \\
3 & -2 & 3 & 0 & 3 & 0 & 2 & 1 & 0 &
-1 & -2 & -2 & -3 & -3 & -2 & -2 & -1 &
0 & 1 & 2 & 0 & 3 \\
0 & 3 & -2 & 3 & 0 & 3 & 0 & 2 & 1 &
0 & -1 & -2 & -2 & -3 & -3 & -2 & -2 &
-1 & 0 & 1 & 2 & 0 \\
3 & 0 & 3 & -2 & 3 & 0 & 3 & 0 & 2 &
1 & 0 & -1 & -2 & -2 & -3 & -3 & -2 & -2
& -1 & 0 & 1 & 2 \\
0 & 3 & 0 & 3 & -2 & 3 & 0 & 3 & 0 &
2 & 1 & 0 & -1 & -2 & -2 & -3 & -3 & -2
& -2 & -1 & 0 & 1 \\
2 & 0 & 3 & 0 & 3 & -2 & 3 & 0 & 3 &
0 & 2 & 1 & 0 & -1 & -2 & -2 & -3 & -3
& -2 & -2 & -1 & 0 \\
1 & 2 & 0 & 3 & 0 & 3 & -2 & 3 & 0 &
3 & 0 & 2 & 1 & 0 & -1 & -2 & -2 & -3
& -3 & -2 & -2 & -1 \\
0 & 1 & 2 & 0 & 3 & 0 & 3 & -2 & 3 &
0 & 3 & 0 & 2 & 1 & 0 & -1 & -2 & -2
& -3 & -3 & -2 & -2 \\
-1 & 0 & 1 & 2 & 0 & 3 & 0 & 3 & -2
& 3 & 0 & 3 & 0 & 2 & 1 & 0 & -1 &
-2 & -2 & -3 & -3 & -2 \\
-2 & -1 & 0 & 1 & 2 & 0 & 3 & 0 & 3
& -2 & 3 & 0 & 3 & 0 & 2 & 1 & 0 &
-1 & -2 & -2 & -3 & -3 \\
-2 & -2 & -1 & 0 & 1 & 2 & 0 & 3 & 0
& 3 & -2 & 3 & 0 & 3 & 0 & 2 & 1 & 0
& -1 & -2 & -2 & -3 \\
-3 & -2 & -2 & -1 & 0 & 1 & 2 & 0 & 3
& 0 & 3 & -2 & 3 & 0 & 3 & 0 & 2 & 1
& 0 & -1 & -2 & -2 \\
-3 & -3 & -2 & -2 & -1 & 0 & 1 & 2 & 0
& 3 & 0 & 3 & -2 & 3 & 0 & 3 & 0 & 2
& 1 & 0 & -1 & -2 \\
-2 & -3 & -3 & -2 & -2 & -1 & 0 & 1 & 2
& 0 & 3 & 0 & 3 & -2 & 3 & 0 & 3 & 0
& 2 & 1 & 0 & -1 \\
-2 & -2 & -3 & -3 & -2 & -2 & -1 & 0 & 1
& 2 & 0 & 3 & 0 & 3 & -2 & 3 & 0 & 3
& 0 & 2 & 1 & 0 \\
-1 & -2 & -2 & -3 & -3 & -2 & -2 & -1 &
0 & 1 & 2 & 0 & 3 & 0 & 3 & -2 & 3 &
0 & 3 & 0 & 2 & 1 \\
0 & -1 & -2 & -2 & -3 & -3 & -2 & -2 &
-1 & 0 & 1 & 2 & 0 & 3 & 0 & 3 & -2
& 3 & 0 & 3 & 0 & 2 \\
1 & 0 & -1 & -2 & -2 & -3 & -3 & -2 & -2
& -1 & 0 & 1 & 2 & 0 & 3 & 0 & 3 &
-2 & 3 & 0 & 3 & 0 \\
2 & 1 & 0 & -1 & -2 & -2 & -3 & -3 & -2
& -2 & -1 & 0 & 1 & 2 & 0 & 3 & 0 &
3 & -2 & 3 & 0 & 3 \\
0 & 2 & 1 & 0 & -1 & -2 & -2 & -3 & -3
& -2 & -2 & -1 & 0 & 1 & 2 & 0 & 3 &
0 & 3 & -2 & 3 & 0 \\
3 & 0 & 2 & 1 & 0 & -1 & -2 & -2 & -3
& -3 & -2 & -2 & -1 & 0 & 1 & 2 & 0
& 3 & 0 & 3 & -2 & 3 \\
0 & 3 & 0 & 2 & 1 & 0 & -1 & -2 & -2
& -3 & -3 & -2 & -2 & -1 & 0 & 1 & 2
& 0 & 3 & 0 & 3 & -2
\end{array}\right)
$
\label{tab:matrice23}
\end{rotate}
\end{center}

\newpage

\bibliographystyle{amsplain}
\bibliography{Biblio}

\end{document}